\documentclass[11pt,a4paper]{article}
\usepackage[latin1]{inputenc}
\usepackage[american]{babel}      % language
\usepackage[T1]{fontenc}
\usepackage{amsmath,amssymb,amsthm}
\usepackage{graphicx}

\title{An inverse result of approximation by\\ sampling Kantorovich series}
         
\author{{\bf Danilo Costarelli} \hskip0.2cm and \hskip0.2cm {\bf Gianluca Vinti}     \\    \\
   Department of Mathematics and Computer Science \\
            University of Perugia\\
        1, Via Vanvitelli, 06123 Perugia, Italy    \\  \\
  {\small {\tt danilo.costarelli@unipg.it} \hskip0.2cm and \hskip0.2cm {\tt gianluca.vinti@unipg.it}} }

\date{}

\newcommand{\mau}{\geq}
\newcommand{\miu}{\leq}
\newcommand{\ep}{\varepsilon}
\newcommand{\N}{\mathbb{N}}
\newcommand{\R}{\mathbb{R}}
\newcommand{\Z}{\mathbb{Z}}

\newcommand{\be}{\begin{equation}}
\newcommand{\ee}{\end{equation}}

\newcommand{\F}{F_{\Phi}}

\newtheorem{definition}{Definition}[section]
\newtheorem{remark}[definition]{Remark}
\newtheorem{theorem}[definition]{Theorem}
\newtheorem{lemma}[definition]{Lemma}

\begin{document}

\maketitle  

\begin{abstract}
In the present paper, an inverse result of approximation, i.e., a saturation theorem for the sampling Kantorovich operators is derived, in the case of uniform approximation for uniformly continuous and bounded functions on the whole real line. In particular, here we prove that the best possible order of approximation that can be achieved by the above sampling series is the order one, otherwise the function being approximated turns to be a constant. The above result is proved by exploiting a suitable representation formula which relates the sampling Kantorovich series with the well-known generalized sampling operators introduced by P.L. Butzer. At the end, some other applications of such representation formula are presented, together with a discussion concerning the kernels of the above operators for which such an inverse result occurs.
\end{abstract}

\section{Introduction}

The theory of sampling-type operators has been largely studied since 1980s, when, in order to provide an approximate version of the classical Whittaker-Kotel'nikov-Shannon sampling theorem (see e.g. \cite{HI1,KITA1}), P.L. Butzer introduced the generalized sampling operators $G_w$ (see (\ref{GEN}) of Section \ref{sec2}), and studied their main properties, see e.g., \cite{BUST1,RIST1}. The operators $G_w$ allows to reconstruct (in some sense) a given continuous signal $f$ by a sequence of its sample values, which are of the form $f(k/w)$, $k \in \Z$, $w>0$.  Subsequently, such operators have been widely studied by many authors, see e.g., \cite{BABUSTVI3,DOLU1,TA1,KITA2,TA2,COGA1,COGA2,COGA3}.

In 2007, an $L^1$-version of the above operators have been introduced, with the definition of the sampling Kantorovich series $S_w$ (see (\ref{KANTO}) of Section \ref{sec2}; \cite{BABUSTVI2}), obtained by replacing in $G_w$ the sample values by the mean values $w\int_{k/w}^{(k+1)/w}f(u)\, du$, for any locally integrable signal $f$. The main advantage that can be achieved by the operators $S_w$ with respect to $G_w$, is that also not necessarily continuous signal can be approximated. Multivariate extensions of the above theory have been given in \cite{COVI1,COVI2}. In the latter case, applications to digital image for earthquake engineering have been studied in \cite{CLCOMIVI1,CLCOMIVI2,ING1}. Both the operators $G_w$ and $S_w$ are based upon suitable kernel functions satisfying certain assumptions. 

  Recently, the sampling Kantorovich operators have been studied with respect to various aspects, e.g., the convergence in suitable spaces of function (see e.g., \cite{VIZA1}), the order of approximation (see e.g., \cite{COVI3,COVI4,ORTA1}), their behavior at the discontinuity points of a given signal (see \cite{COMIVI1}), and so on. However, the following inverse problem is still open:

  if there exist a positive non-increasing function $\varphi(w)$, $w \in \R^+$, with $\lim_{w \to+\infty} \varphi(w)=0$, and a class of functions $\mathcal{K} \subseteq C(\R)$ (the space of uniformly continuous and bounded functions) such that:
$$
\mbox{(I)} \hskip0.4cm \| S_wf - f\|_{\infty}\, =\, o\left(\varphi(w)\right), \hskip0.2cm as\ w \to +\infty, \hskip0.2cm \mbox{implies}\, f=constant.
$$

  The main result showed in the present paper consists in proving (I) with $\mathcal{K} = C(\R)$, and $\varphi(w)=1/w$, i.e., we prove that the best possible order of approximation that can be achieved by the sampling Kantorovich operators is ``one''. The main steps required in order to prove the above result, are the following. Firstly we prove a representation formula for the sampling Kantorovich series in terms of the generalized sampling operators of $f$ and its derivatives until order $r$, provided that they exist, and are all uniformly continuous and bounded, namely $f$ belongs to the class $C^{(r)}(\R)$, $r \in \N^+$. Subsequently, we obtain a saturation result for the subspace $C^{(2)}(\R)$. Finally, we consider functions in $C(\R)$, and by the regularization provided by the convolution with suitable test functions, we become able to prove a version of the desired result (I) by exploiting the inverse results for $C^{(2)}$-functions (see Section \ref{sec3}). 
  
  The solution of the problem (I) can open the way to obtain a characterization of the saturation (Favard) classes of the approximation process defined by the sampling Kantorovich operators.

Note that, the inverse result just discussed, is quite different with respect to what happens in case of operators $G_w$ where, in order to obtain similar result, we need to require that $f \in C^{(r)}(\R)$, $r \in \N^+$, and therefore our problem cannot avail of the result for $G_w$.

  In conclusion of Section \ref{sec3}, we prove a further consequence of the above representation formula, by showing that under suitable assumptions on the kernels, the sampling Kantorovich operators maps algebraic polynomials into other polynomials with the same degree. Examples of kernels for which the above results hold are provided in Section \ref{sec4}.

%%%%%%%%%%%%%%%%%%%

\section{Preliminaries} \label{sec2}

We firstly introduce some notations. In what follows, for any arbitrary finite or infinite interval $I \subseteq \R$, we denote by $C(I)$ the space of all uniformly continuous and bounded functions $f:I \to \R$, endowed with the supremum norm $\|f\|_{\infty}:=\sup_{x \in I}|f(x)|$. Further, we denote by $C^{(r)}(I)$, $r \in \N^+$ the subspace of $C(I)$ for which the derivatives $f^{(s)}$ exist, for every $s \miu r$, $s \in \N^+$, and each $f^{(s)} \in C(I)$. Moreover, we define by $C_{c}(I)$ the subspace of $C(I)$ of functions having compact support, and similarly we can define $C^{(r)}_{c}(I)$, $r \in \N^+$. Finally, by $C^{\infty}_{c}(I)$ we denote the space of test functions, i.e., the space of functions with compact support which have continuous derivatives of any order, each one belonging to $C_{c}(I)$.

For any $f:\R \to \R$, we can define the discrete moment of order $\beta \in \N$, at point $u \in \R$, by:
\be
m_{\beta}(f,\, u)\ :=\ \sum_{k \in \Z} f(u-k)\, (u-k)^{\beta}, 
\ee
and the discrete absolute moment of order $\beta \mau 0$, by:
\be
M_{\beta}(f)\ :=\ \sup_{u \in \R}\, \sum_{k \in \Z} |f(u-k)|\, \cdot |u-k|^{\beta}.
\ee
Clearly, if the function $f$ belongs to $C_{c}(\R)$, it turns out that both $m_{\beta}(f,\, u)$, and $M_{\beta}(f)$ are finite, respectively for every $\beta \in \N$, with $u \in \R$, and $\beta \mau 0$.

  Now, we are able to recall the definition of the sampling Kantorovich operators, introduced in \cite{BABUSTVI2}:
\be \label{KANTO}
(S_w f)(x)\ :=\ \sum_{k \in \Z} \chi\left( wx-k\right)\, \left[ w\, \int_{k/w}^{(k+1)/w}f(u)\, du \right], \hskip1cm x \in \R,
\ee
where $f:\R \to \R$ is a locally integrable function, such that the above series is convergent for every $x \in \R$, and $\chi: \R \to \R$ is a kernel, i.e., a function which satisfies the following assumptions:
\begin{itemize}
\item[$(\chi 1)$] $\chi$ belongs to $L^1(\R)$ and it is locally bounded at the origin;
\item[$(\chi 2)$] the series $\sum_{k \in \Z} \chi(u-k)=1$, for every $u \in \R$;
\item[$(\chi 3)$] there exists $\beta>0$ for which $M_{\beta}(\chi)$ is finite.
\end{itemize}
Note that, in general it is possible to prove that, $(\chi 3)$ implies that $M_{\nu}(\chi)$ is finite, for every $0\miu \nu \miu \beta$, see e.g., \cite{BABUSTVI2,COVI3,COVI4}.

For instance, if we assume that $f \in L^{\infty}(\R)$, it turns out that $S_w f \in L^{\infty}(\R)$, i.e., $S_w$ maps $L^{\infty}(\R)$ to $L^{\infty}(\R)$, see \cite{BABUSTVI2}.

 Moreover, under the assumptions $(\chi i)$, $i=1, 2, 3$, the family of the sampling Kantorovich series $S_w f$ converges to $f$ pointwise at $x \in \R$, as $w \to +\infty$, provided that $f$ is bounded and continuous at $x$; the convergence is uniform on $\R$, if $f$ belongs to $C(\R)$, see \cite{BABUSTVI2} again.

  We recall that, the sampling Kantorovich operators have been introduced in order to provide an $L^1$-version of the classical generalized sampling operators, which are defined by:
\be \label{GEN}
(G_wf)(x)\ :=\ \sum_{k \in \Z} \chi\left( wx-k\right)\,  f\left( \frac{k}{w} \right), \hskip1cm x \in \R,
\ee 
with $w>0$, and where $\chi$ is a kernel satisfying assumptions $(\chi i)$, $i=1,2,3$. Both the operators $S_w$ and $G_w$ are instances of ``quasi-interpolation'' operators, see e.g., \cite{GAWU1,LIWA1,SP1,WAYUZH1}.

Pointwise and uniform convergence results analogous to those proved for the sampling Kantorovich series can be proved also for $(G_w f)_{w>0}$, see e.g., \cite{BUST1}. Now, we recall also the following high order convergence result, which can be useful in the present paper. 
\begin{theorem}[\cite{BUST1}] \label{samp-gen}
Let $\chi$ be a kernel, which satisfies the following condition:
\be \label{null_moments}
m_{j}\left(\chi,\, u\right)\ :=\ \left\{
\begin{array}{l}
0, \hskip1cm j=1, 2, ..., r-1, \\
1, \hskip1cm j=r,
\end{array}
\right.
\ee
for every $u \in \R$, and some $r \in \N^+$.

\noindent Then, for any $f \in C^{(r)}(\R)$ it holds:
$$
\| G_w f - f\|_{\infty}\ \miu\ \|f^{(r)}\|_{\infty}\, \frac{M_r(\chi)}{r!}\,  w^{-r},
$$
for every $w>0$.
\vskip0.2cm
\noindent Moreover, the following property occurs:
$$
\left(G_w p_{r-1}\right)(x)\ =\ p_{r-1}(x),  \hskip1cm x \in \R,
$$
for every $w>0$, where $p_{r-1}(x)$ denotes any algebraic polynomial of degree $r-1$.  
\end{theorem}
Conditions like (\ref{null_moments}) is to be found in connection with finite element approximation, see e.g., \cite{FIST1}. 

 In general, to check if a given kernel $\chi$ satisfies assumption (\ref{null_moments}) can be difficult. For this reason, the following lemma can be useful.
\begin{lemma}[\cite{BUST1}] \label{lemma1}
Let $\chi$ be a continuous kernel. Condition (\ref{null_moments}) is equivalent to the following:
$$
(\widehat{\chi})^{(j)}(2\, \pi k)\ =\ \left\{
\begin{array}{l}
1, \hskip1cm k=j=0,\\
0, \hskip1cm k\in \Z\setminus\left\{ 0 \right\}, \hskip0.5cm j=0,\\
0, \hskip1cm k\in \Z, \hskip1.5cm j=1,2, ..., r-1,
\end{array}
\right.
$$
where $\widehat{\chi}(v) :=   \int_{\R} \chi(u)\, e^{-i\, u\, v}\, du$, $v \in \R$, denotes the Fourier transform of $\chi$.
\end{lemma}

  Note that, for the sake of completeness, a high order approximation theorem for the sampling Kantorovich operators, analogous to the above, cannot be proved, see e.g., \cite{BAMA1,BAMA1.1,BAMA2}. Moreover, the rate of convergence for the family $(S_w f)_{w>0}$ has been also studied in \cite{COVI3,COVI4,COVI5} in $C(\R)$, and in the Orlicz spaces $L^{\varphi}(\R)$, by considering functions in suitable Lipschitz classes.
%

%%%%%-----------------------------------------------

\section{Inverse result}  \label{sec3}

In order to prove an inverse result for the sampling Kantorovich series, we needs the following representation formula, which allows to state the relation between the operators $S_w f$ and $G_w f$, when functions belonging to $C^{(r)}(\R)$, $r \in \N^+$, are considered.  
\begin{theorem}  \label{th1}
For any $f \in C^{(r)}(\R)$, $r \in \N^+$, it holds:
$$
(S_wf)(x)\ =\ \sum_{j=0}^{r-1} \frac{w^{-j}}{(j+1)!}\, \left( G_w f^{(j)} \right)(x)\ +\ \mathcal{R}^w_r(x), \hskip1cm x \in \R,
$$
where the remainder of order $r$, is the following absolutely convergent series:
$$
\mathcal{R}^w_r(x)\ :=\ \frac{1}{r!}\, \sum_{k \in \Z}w\, \left[ \int_{k/w}^{(k+1)/w}\!\!\! f^{(r)}\left( \theta_{k, w}(u)\right)\cdot \left(u- k/w \right)^r\, du \right] \chi(wx-k),
$$
where $\theta_{k, w}(u)$ are measurable functions, such that $k/w < \theta_{k, w}(u) < (k+1)/w$, $k \in \Z$, for every $u \in [k/w, (k+1)/w]$, $w>0$.
\end{theorem}
\begin{proof}
By considering the Taylor formula with the Lagrange remainder, applied to $f$, we have:
$$
f(u)\ =\ \sum_{j=0}^{r-1} \frac{f^{(j)}(x)}{j!}\, (u-x)^j\ +\ \frac{f^{(r)}(\theta_{u,x})}{r!}\, (u-x)^r,
$$
for $x$, $u \in \R$, and $\theta_{u,x} \in (x,u)$. Now, if we set $x=k/w$, $k\in \Z$ and $w>0$, in the above formula, for every $u \in ( k/w, (k+1)/w]$ it turns out that $k/w<\theta_{u,k/w}=:\theta_{k,w}(u)<(k+1)/w$. Then, replacing the above expansion with $x=k/w$ in the integrals $w \int_{k/w}^{(k+1)/w}f(u)\, du$, we can write what follows:
\vskip0.1cm
$$
\hskip-7cm w \int_{k/w}^{(k+1)/w} f(u)\, du\ 
$$
$$
=\ w \int_{k/w}^{(k+1)/w}\left[ \sum_{j=0}^{r-1} \frac{f^{(j)}(k/w)}{j!}\, (u-k/w)^j\ +\ \frac{f^{(r)}(\theta_{k,w}(u))}{r!}\, (u-k/w)^r \right]\, du
$$
$$
=\, w\, \sum_{j=0}^{r-1} \frac{f^{(j)}(k/w)}{j!}\int_{k/w}^{(k+1)/w}\!\!\!\!(u-k/w)^j\ du\, +\, \frac{w}{r!} \int_{k/w}^{(k+1)/w}\!\!\!\!\!f^{(r)}(\theta_{k,w}(u))\, (u-k/w)^r\, du
$$
\be \label{medie_espanse}
\hskip-1.9cm =\ \sum_{j=0}^{r-1} \frac{f^{(j)}(k/w)}{(j+1)! }\, w^{-j} +\ \frac{w}{r!} \int_{k/w}^{(k+1)/w}f^{(r)}(\theta_{k,w}(u))\, (u-k/w)^r\, du.
\ee
Now, by exploiting (\ref{medie_espanse}) in the definition of $(S_w f)(x)$, $x \in \R$, we obtain:
$$
\hskip-8.5cm (S_w f)(x)\ =\ 
$$
$$
= \sum_{k \in \Z}\chi(wx-k)\! \left[ \sum_{j=0}^{r-1} \frac{f^{(j)}(k/w)}{(j+1)! }\, w^{-j} + \frac{w}{r!} \int_{k/w}^{(k+1)/w}\!\!\!f^{(r)}(\theta_{k,w}(u))\, (u-k/w)^r\, du \right]
$$
$$
\hskip-6cm =\ \sum_{j=0}^{r-1}\frac{w^{-j}}{(j+1)! }\,  \sum_{k \in \Z}\chi(wx-k)\, f^{(j)}\left(\frac{k}{w}\right)\ 
$$
$$
\hskip2.5cm +\ \frac{1}{r!}\, \sum_{k \in \Z}\chi(wx-k)\! \left[w\,  \int_{k/w}^{(k+1)/w}\!\!\!f^{(r)}(\theta_{k,w}(u))\, (u-k/w)^r\, du \right]
$$
$$
\hskip-6cm =\ \sum_{j=0}^{r-1}\frac{w^{-j}}{(j+1)! }\,  \left( G_w f^{(j)}\right)(x)\ +\ {\cal R}^w_r(x),
$$
for every $w>0$, where:
$$
{\cal R}^w_r(x)\ :=\ \frac{1}{r!}\, \sum_{k \in \Z}\chi(wx-k)\! \left[w\,  \int_{k/w}^{(k+1)/w}\!\!\!f^{(r)}(\theta_{k,w}(u))\, (u-k/w)^r\, du \right].
$$ 
Note that, the series ${\cal R}^w_r(x)$ is absolutely convergent for every $x \in \R$, for every $w>0$. Indeed,
$$
\frac{1}{r!}\, \sum_{k \in \Z}|\chi(wx-k)|\! \left|w\,  \int_{k/w}^{(k+1)/w}\!\!\!f^{(r)}(\theta_{k,w}(u))\, (u-k/w)^r\, du \right| 
$$
\be \label{ord-r}
\miu\ \frac{\|f^{(r)} \|_{\infty}}{(r+1)!}\, w^{-r}\, \sum_{k \in \Z}|\chi(wx-k)|\ \miu\ \frac{\|f^{(r)} \|_{\infty}}{(r+1)!}\, w^{-r}\, M_0(\chi)\ <\ +\infty.
\ee
This completes the proof.
\end{proof}
\begin{remark} \rm
Note that, by (\ref{ord-r}) easily follows that the remainder ${\cal R}^w_r(x)$ in the representation formula of Theorem \ref{th1}, is such that:
$$
{\cal R}^w_r(x)\ =\ {\cal O}(w^{-r}), \hskip1cm \mbox{as} \hskip0.5cm w \to +\infty,
$$
for every $x \in \R$. 
\end{remark}
Now, we can state the main result of this section.
\begin{theorem} \label{th2}
Let $\chi$ be a kernel, which satisfies the moment condition (\ref{null_moments}), for every $u \in \R$, with $r = 2$. Now, let $f \in C(\R)$, and suppose in addition that:
\be \label{opiccolo-uniform}
\| S^\pi_w f - f \|_{\infty}\ =\ o(w^{-1}), \hskip1cm as \hskip0.5cm w \to +\infty,
\ee
uniformly with respect to every sequence $\pi = (t_k)_{k \in \Z} \subset \R$, such that $\lim_{k \to \pm \infty} t_k=\pm \infty$, with $t_{k+1}-t_k=1$, $k \in \Z$, and where:
$$
(S^\pi_w f)(x)\ :=\ \sum_{k \in \Z}\left[ w \int_{t_k/w}^{t_{k+1}/w} f(u)\, du\right] \chi(wx-t_k), \hskip1cm x \in \R.
$$
Then, $f$ is constant over $\R$.
\end{theorem}
Note that, assumption (\ref{opiccolo-uniform}) which involves the operators (\ref{KANTO}) for a general sampling scheme $\pi=(t_k)_{k \in \Z} \subseteq \R$, is meaningful and not restrictive, in view of the results concerning the order of approximation proved in \cite{COVI3}, for the series $S^\pi_w$.

Moreover, we also point out that, to prove the above theorem, it is sufficient that the assumptions on $\chi$ are satisfied for the sequence $t_k=k$, $k \in \Z$ (as in the form given in Section \ref{sec2}).
 
In order to obtain the proof of Theorem \ref{th2}, we firstly prove the above result for functions belonging to $C^{(2)}(\R)$. We have the following.
\begin{theorem} \label{thC2}
Let $\chi$ be a kernel, which satisfies the moment condition (\ref{null_moments}), for every $u \in \R$, with $r = 2$. Now, let $f \in C^{(2)}(\R)$, and suppose that:
\be \label{opiccolo}
\| S_w f - f \|_{\infty}\ =\ o(w^{-1}), \hskip1cm as \hskip0.5cm w \to +\infty.
\ee
Then, it turns out that $f$ is constant on $\R$.
\end{theorem}
\begin{proof}
Since $f$ belongs to $C^{(2)}(\R)$, the representation formula of Theorem \ref{th1} can be applied, e.g., until order $r=1$, i.e., for every $x \in \R$ we can write:
$$
(S_w f)(x)\ =\ (G_w f)(x)\ +\ {\cal R}^w_1(x),
$$ 
for every $w>0$, then assumption (\ref{opiccolo}) can be rewritten as follows:
$$
|(G_w f)(x)\ +\ {\cal R}^w_1(x) - f(x)|\ =\ o(w^{-1}), \hskip1cm as \hskip0.5cm w \to +\infty,
$$
i.e., 
$$
\lim_{w\to +\infty} w \left[  (G_w f)(x)\ +\ {\cal R}^w_1(x)\ -\ f(x)   \right]\ =\ 0,
$$
for every $x \in \R$. Now splitting the above limit (since as we will show below they exist and are both finite), we can write:
\be \label{hhhjkl}
\lim_{w\to +\infty} w \left[  (G_w f)(x)\ -\ f(x)\right]\ +\ \frac{1}{2}\, \lim_{w\to +\infty} 2\, w\, {\cal R}^w_1(x)\ =\ 0.
\ee
Now, since (\ref{null_moments}) is satisfied for $r=2$, in view of Theorem \ref{samp-gen} we know that $\| G_w f - f\|={\cal O}(w^{-2})$, as $w\to +\infty$, then it is easy to see that:
$$
\lim_{w\to +\infty} w \left[  (G_w f)(x)\ -\ f(x)\right]\ =\ 0,
$$ 
so we can deduce from (\ref{hhhjkl}) that:
\be \label{mjbgsxh}
\lim_{w\to +\infty} 2\, w\, {\cal R}^w_1(x)\ =\ 0.
\ee
Now, we claim that the family $(2\, w\, {\cal R}^w_1)_{w>0}$ converges uniformly (then also pointwise) to $f'$ on $\R$.
In order to prove the above statement, we proceed by estimating:
$$
\hskip-8cm \left| 2\, w\, {\cal R}^w_1(x)\ -\ f'(x) \right|\ 
$$
$$
\miu\ \left| 2\, w\, {\cal R}^w_1(x)\ -\ (G_wf')(x) \right|\ +\ \left| (G_wf')(x)\ -\ f'(x) \right|\ =:\ I_1\ +\ I_2,
$$
$w>0$. Let now $\ep>0$ be fixed. Since $f'$ is uniformly continuous and bounded, by the well-know convergence results concerning the generalized sampling series, we immediately have that $I_2 < \ep$, for sufficiently large $w>0$, see e.g., \cite{BUST1,BABUSTVI3}. Now, we estimate $I_1$. We can write what follows:
$$
I_1\ \miu\ \sum_{k \in \Z} \left|2\, w^2\,  \int_{k/w}^{(k+1)/w}\!\! f'(\theta_{k,w}(u))\, (u-k/w)\, du\ -\ f'(k/w)   \right|\, |\chi(wx-k)|.
$$
For any $k \in \Z$, and sufficiently large $w>0$ we have:
$$
\hskip-3cm \left|2\, w^2\,  \int_{k/w}^{(k+1)/w}\!\! f'(\theta_{k,w}(u))\, (u-k/w)\, du\ -\ f'(k/w) \right|
$$
$$
=\ \left|2 w^2  \int_{k/w}^{(k+1)/w}\!\!\!\! f'(\theta_{k,w}(u))\, (u-k/w)\, du - 2 w^2 f'(k/w)\int_{k/w}^{(k+1)/w}\!\!\!\!(u-k/w)\, du \right|
$$
\be \label{bohh1}
\miu\ 2\, w^2\, \int_{k/w}^{(k+1)/w}\left|f'(\theta_{k,w}(u))-f'(k/w)\right|\, (u-k/w)\, du,
\ee
where $k/w<\theta_{k,w}(u)<(k+1)/w$. Now, since $f'$ is uniformly continuous, and $\theta_{k,w}(u)-k/w \miu 1/w$, we have that, in correspondence of $\ep>0$, 
\be \label{unif-cont}
\left|f'(\theta_{k,w}(u))-f'(k/w)\right|\ <\ \ep,
\ee
for sufficiently large $w>0$. Now, replacing (\ref{unif-cont}) in (\ref{bohh1}) we finally obtain :
$$
\left|2\, w^2\,  \int_{k/w}^{(k+1)/w}\!\! f'(\theta_{k,w}(u))\, (u-k/w)\, du\ -\ f'(k/w) \right|\ <\ \ep.
$$
In conclusion, we have:
$$
I_1\ \miu\ \ep\, \sum_{k \in \Z} |\chi(wx-k)|\ \miu\ \ep\, M_0(\chi),
$$
for $w>0$ sufficiently large, then the above claim is now proved, i.e., 
\be \label{kiiooo}
\lim_{w \to +\infty} 2\, w\, {\cal R}^w_1(x)=f'(x)
\ee 
for every $x \in \R$. Then, in view of (\ref{mjbgsxh}) and (\ref{kiiooo}) we obtain that $f'(x)=0$, for every $x \in \R$, i.e, $f$ is constant on the whole $\R$.
\end{proof}
Now, we are able to provide the proof of Theorem \ref{th2}. 
\begin{proof}[Proof of Theorem \ref{th2}]
Let $f \in C(\R)$ be fixed, such that (\ref{opiccolo-uniform}) is satisfied. Moreover, let $\Phi \in C^{\infty}_{c}(\R)$ be a test function. We denote by:
$$
F_{\Phi}(x)\ :=\ (\Phi * f)(x)\ =\ \int_{\R} \Phi(x-t)\, f(t)\, dt, \hskip1cm x \in \R,
$$
where ``$*$'' denotes the usual convolution product. Note that, $\F(x)$ is well-defined since $f$ is continuous then it belongs to $L^1_{Loc}(\R)$, and in view of the regularization properties of ``$*$'', it turns out that $\F$ belongs, e.g., to $C^{(2)}(\R)$. Indeed, it is easy to see that both the first and the second derivative of $\F$ are uniformly continuous, together with $\F$ itself, in view of the uniform continuity of $f$. Now, for every fixed $x \in \R$, by exploiting condition $(\chi 2)$ and Fubini-Tonelli theorem, we can write what follows:
$$
(S_w \F)(x) - \F(x)\ =\ \sum_{k \in \Z}\left\{ w \int_{k/w}^{(k+1)/w} [\F(u)-\F(x)]\ du\right\}\, \chi(wx-k)
$$
$$
=\  \sum_{k \in \Z}\left\{ w \int_{k/w}^{(k+1)/w} \left[\int_{\R}\Phi(u-t)f(t)\, dt\, -\, \int_{\R}\Phi(x-t)f(t)\, dt\right] du\right\} \chi(wx-k)
$$
$$
=\  \sum_{k \in \Z}\left\{ w \int_{k/w}^{(k+1)/w} \left[\int_{\R}\Phi(y)f(x-y)\, dy - \int_{\R}\Phi(y)f(u-y)\, dy\right] du\right\} \chi(wx-k)
$$
$$
\hskip-0.4cm =\ \sum_{k \in \Z}\left\{ w \int_{k/w}^{(k+1)/w} \left( \int_{\R}\Phi(y)\left[f(x-y)\,-\, f(u-y)\right]dy\right) du  \right\} \chi(wx-k)
$$
$$
=\ \sum_{k \in \Z}\int_{\R}\Phi(y) \left\{  w  \left( \int_{k/w}^{(k+1)/w}\left[f(x-y)\,-\, f(u-y)\right]du\right)\ \chi(wx-k) \right\} dy.
$$
Now, if we set:
$$
\sum_{k \in \Z}\Phi(y) \left\{  w  \left( \int_{k/w}^{(k+1)/w}\left[f(x-y)\,-\, f(u-y)\right]du\right)\ \chi(wx-k) \right\}=: \sum_{k \in \Z}h_k(y),
$$
we have that the above series is absolutely convergent (hence also convergent) for every $y \in \R$, since:
$$
\sum_{k \in \Z}|h_k(y)|\ \miu\ 2\|\Phi\|_{\infty}\, \|f\|_{\infty}\, M_0(\chi)\ <\ +\infty,
$$
and moreover, for every $n \in \N^+$:
$$
\left|\sum_{k=-n}^n h_k(y)\right|\ \miu\ 2\, \|f\|_{\infty}\, M_0(\chi)\, |\Phi(y)|\ =:\ H(y),\, \hskip1cm y \in \R,
$$
with $H \in L^1(\R)$. Then, by the Lebesgue dominated convergence theorem, we can write:
$$
\hskip-8cm (S_w \F)(x)\ -\ \F(x)
$$
$$
\hskip-0.1cm=\ \int_{\R}\Phi(y)\, \left( \sum_{k \in \Z}\left\{ w \int_{k/w}^{(k+1)/w}\left[f(x-y)\,-\, f(u-y)\right]du  \right\} \chi(wx-k)\, \right) dy.
$$
Now, by setting $g_y(x):=f(x-y)$, for every $x \in \R$, and $y \in \R$, by using $H\ddot{o}lder$ inequality, we obtain:
$$
\hskip-1.4cm \left|(S_w \F)(x) - \F(x)\right|\ =\ \left|\int_{\R}\Phi(y)\, \left[g_y(x)- (S_w g_y)(x) \right]\, dy\right|
$$
\be \label{piripi}
\hskip0.8cm \miu \int_{\R}\left|\Phi(y)\right|\, \left| (S_w g_y)(x) - g_y(x) \right|\, dy \miu \| \Phi\|_1\, \| S_w g_{(\cdot)} - g_{(\cdot)}\|_{\infty},
\ee
for every $x \in \R$, where:
\be \label{norma}
 \| (S_w g_{(\cdot)})(x)-g_{(\cdot)}(x) \|_{\infty}= \sup_{y\in\mathbb{R}} \left| (S_w g_{y})(x)-g_{y}(x) \right|,  
\ee
for fixed $x$ and $w$. Now, using respectively the changes of variables $u-y=t$ and $k-yw=:t_{k}^{(y, w)},\ k\in\mathbb{Z}$, we obtain:
$$
\hskip-2.0cm (S_w g_{y})(x)=\displaystyle{\sum\limits_{k\in\mathbb{Z}}\left[ w\int_{\frac{k}{w}}^{\frac{k+1}{w}}g_{y}(u)\, du \right] \chi(wx-k)}
$$
$$
=\displaystyle{\sum\limits_{k\in\mathbb{Z}}\left[ w\int_{\frac{k}{w}}^{\frac{k+1}{w}}f(u-y)\, du \right] \chi(wx-k)}
$$
$$
=\displaystyle{\sum\limits_{k\in\mathbb{Z}}\left[ w\int_{\frac{k}{w}-y}^{\frac{k+1}{w}-y}f(t)\, dt \right] \chi(wx-k)}
$$
$$
=\displaystyle{\sum\limits_{k\in\mathbb{Z}}\left[ w\int_{\frac{k-yw}{w}}^{\frac{k+1-yw}{w}}f(t)\, dt \right] \chi(wx-k)}
$$
$$
=\displaystyle{\sum\limits_{k\in\mathbb{Z}}\left[ w\int^{\frac{t^{(y, w)}_{k}+1}{w}}_{\frac{t^{(y, w)}_{k}}{w}}f(t)\, dt \right] \chi\left(wx-(t^{(y, w)}_{k}+yw)\right)}
$$
$$
=\displaystyle{\sum\limits_{k\in\mathbb{Z}}\left[ w\int_{\frac{t^{(y, w)}_{k}}{w}}^{\frac{t^{(y, w)}_{k}+1}{w}}f(t)dt \right] \chi\left(w(x-y)-t^{(y, w)}_{k})\right)=\left(S_w^{\pi^w_y}f\right)(x-y)},
$$
where $\pi^w_y=(t^{(y, w)}_{k})_{k\in\mathbb{Z}}$, for every $y\in\mathbb{R}$. Now, it is easy to observe that $\lim_{k \to \pm \infty}t^{(y, w)}_{k} = \pm \infty$, and:
$$
t^{(y, w)}_{k+1}-t^{(y, w)}_{k}\ =\ k+1-yw -k + yw\ =\ 1,
$$
for every $k \in \Z$. Hence, (\ref{norma}) becomes:
$$
\sup_{y\in\mathbb{R}} \left|(S_wg_{y})(x)-g_{y}(x) \right|\ =\ \sup_{y\in\mathbb{R}} \left| (S_w^{\pi^w_y}f)(x-y)-f(x-y) \right|.
$$
In view of the above equality, since all the sequences of the form $\pi^w_y$, $y \in \R$, $w>0$, satisfy the conditions required in assumption (\ref{opiccolo-uniform}), and $\| \Phi\|_1<+\infty$, using (\ref{piripi}) we finally have:
$$
\|S_w \F - \F\|_{\infty}\ =\ o(w^{-1}), \hskip1cm as \hskip0.5cm w \to +\infty,
$$
for every test function $\Phi \in C^{\infty}_{c}(\R)$. We have proved that any $\F$ satisfies the assumptions of Theorem \ref{thC2}, then it turns out that $\F(x)=k$, for every $x \in \R$, for a suitable constant $k \in \R$.
Thus, for every $x \in \R$ we have:
$$
0\ =\ \F(x)\ -\ \F(0)\ =\ \int_{\R}\Phi(x-t)f(t)\, dt\ -\ \int_{\R}\Phi(-t)f(t)\, dt
$$
$$
=\ \int_{\R}\Phi(y)f(-y)\, dy\ -\ \int_{\R}\Phi(y)f(x-y)\, dy\ =\ \int_{\R}\Phi(y)\, [ f(-y) - f(x-y)]\, dy,
$$
where the equality:
$$
\int_{\R}\Phi(y)\, [ f(-y) - f(x-y)]\, dy\ =\ 0,
$$
holds for every test function $\Phi \in C^{\infty}_{c}(\R)$.

Now, in order to conclude the proof, we suppose by contradiction that $f$ is not constant on $\mathbb{R}$, i.e., that there exists $x_0<y_0$ such that $f(x_0) \neq f(y_0)$. Let now $\widetilde{x} \in \mathbb{R}$ such that $\widetilde{x} + y_0 = x_0$, and let $n \in \mathbb{N}^+$ sufficiently large, such that $y_0 \in I_n:=(-n, n)$ (then also $-y_0 \in I_n$). Then, for every $\Phi\in C_c^{\infty}(I_n)$, we have:
\begin{equation*}
\displaystyle{\int_{-n}^n\Phi(y)[f(-y)-f(\widetilde{x}-y)]dy=\int_{\mathbb{R}}\widetilde{\Phi}(y)[f(-y)-f(\widetilde{x}-y)]dy=0,}
\end{equation*}
where $\widetilde{\Phi}$ denotes the zero-extension of $\Phi$ to the whole $\mathbb{R}$. Since the above equality holds for every $\Phi\in C_c^{\infty}(I_n)$, and $f$ is continuous on $\mathbb{R}$, it turns out that (see \cite{B1}):
\begin{equation*}
\displaystyle{f(-y)-f(\widetilde{x}-y)=0, \ \ \ \ \ y\in (-n,n).}
\end{equation*}
Now, setting $y=-y_0$ in the above equality, we finally obtain:
$$
f(y_0) = f(\widetilde{x}+y_0) = f(x_0),
$$
which is a contradiction. This completes the proof.    
\end{proof}
In conclusion of this section, we prove a further nice properties of the sampling Kantorovich operators, that can be deduced from the representation formula achieved in Theorem \ref{th1}.
\begin{theorem}
Let $\chi$ be a kernel satisfying assumption (\ref{null_moments}) with $r \in \N^+$. Then:
$$
(S_w p_{r-1})(x)\ =\ \sum_{j=0}^{r-1}\frac{w^{-j}}{(j+1)!}\, p^{(j)}_{r-1}(x),
$$
for every $w>0$, and for any algebraic polynomials of degree at most $r-1$, i.e., $S_w$ maps algebraic polynomials of degree at most $r-1$ into algebraic polynomials of the same degree.
\end{theorem}
\begin{proof}
The proof follows immediately from the representation formula of Theorem \ref{th1}, the applications of Theorem \ref{samp-gen}, and finally observing that $p^{(r)}_{r-1}(x)=0$, for every $x\in \R$.
\end{proof}
%

%%%%%%----------------------------------------

\section{The construction of the kernels} \label{sec4}

In Section \ref{sec2}, the definition of kernel for the sampling Kantorovich operators $S_w$ (and also for $G_w$) has been provided. Several examples of well-known functions $\chi$ which satisfy assuptions $(\chi 1)$, $(\chi 2)$, and $(\chi 3)$ are given e.g., in \cite{BUNE,BABUSTVI2,COVI1,COVI7,COVI9}.  

  For instance, we can choose as kernels the following one-dimensional band-limited functions:
$$
\hskip-3.2cm F(x)\ :=\ \frac{1}{2}\, \left(\frac{\sin(\pi x /2)}{\pi x / 2}\right)^2, \hskip0.9cm \mbox{(Fej\'er's kernel),}
$$
$$
\hskip-0.2cm V(x)\ :=\ \frac{3}{2\pi}\, \frac{\sin(x /2)\, \sin(3 x /2)}{3 x^2 / 4}, \hskip0.3cm \hskip0.3cm \mbox{(de la Vall\'ee Poussin's kernel),}
$$
$$
\hskip-6.7cm \chi(x)\ :=\ \frac{\sin(\pi x /2)\, \sin(\pi x)}{\pi^2 x^2/2}, 
$$
$$
\hskip0.5cm b^{\alpha}(x):=2^{\alpha}\, \Gamma(\alpha+1)\, |x|^{-(n/2)+\alpha}{\cal B}_{(n/2)+\alpha}(|x|), \hskip0.3cm \mbox{(Bochner-Riesz kernels)},
$$
where $\alpha > (n-1)/2$, ${\cal B}_{\lambda}$ is the Bessel function of order $\lambda$ and $\Gamma$ is the Euler function, and finally,
$$
\hskip-3.2cm J_k(x)=c_k\, \mbox{sinc}^{2k}\left(\frac{x}{2k\pi\alpha}\right), \hskip1cm \mbox{(Jackson-type kernels)}
$$
with $k \in \N$, $\alpha \geq 1$, where the normalization coefficients $c_k$ are given by
$$
c_k\ :=\ \left[ \int_{\R} \mbox{sinc}^{2k}\left(\frac{u}{2 k \pi \alpha} \right) \, du \right]^{-1}.
$$ 
Actually, the above examples of kernels can be used to show the convergence of the operators $S_w$ and $G_w$, but they do not satisfy the moment condition (\ref{null_moments}), which we showed to be crucial in order to prove the inverse results of Section \ref{sec3}. 

Hence, here we briefly describe as it is possible to construct examples of kernels satisfying condition (\ref{null_moments}). The most convenient instances can be constructed by using the so-called central B-splines. 

  First of all, we recall that a function $q: I \to \R$ is called a (polynomial) spline of order $n \in \N^+$ (degree $n-1$) with knots $a_1<a_2<...<a_m$ belonging to $I$, if it coincides with a polynomial of degree $n-1$ on each of the intervals $(a_i,a_{i+1})$, $i=1,2,...,m-1$, see e.g., \cite{MA1,MO1,AMAL1}.

The central B-splines of order $n \in \N^+$, are defined by:
\be \label{splines}
 M_n(x)\ :=\ \frac{1}{(n-1)!} \sum^n_{i=0}(-1)^i \binom{n}{i} 
       \left(\frac{n}{2} + x - i \right)^{n-1}_+,     \hskip0.5cm   x \in \R,
\ee
where $(x)_+ := \max\left\{x,0 \right\}$ denotes ``the positive part'' of $x \in \R$, see e.g., \cite{BUNE,UN1}. They have knots at the points $0$, $\pm 1$, $\pm 2$, ..., $\pm n/2$ in case $n$ is even, and at $\pm 1/2$, $\pm 3/2$, ..., $\pm n/2$ in case $n$ is odd, and their support is the compact interval $[-n/2, n/2]$. The Fourier transform of the $M_n$ (see e.g., \cite{NE1}) is:
$$
\widehat{M_n}(v)\ =\ sinc(v/2)^n, \hskip1cm v \in \R.
$$ 
The central B-splines $M_n$ satisfy the assumptions $(\chi 1)$, $(\chi 2)$, and $(\chi 3)$, i.e., $M_n$ are kernels, see e.g., \cite{BABUSTVI2}. 
Now, we have the following classical theorem.
\begin{theorem}[\cite{BUST1}] \label{generation_kernels}
For $r \in \N^+$, $r \mau 2$, let $\ep_0 < \ep_1 < ... < \ep_{r-1}$ be any given real numbers, and let $a_{\mu_r}$, $\mu=0, 1, ..., r-1$, be the unique solutions of the linear system:
$$
\sum_{\mu=0}^{r-1}a_{\mu_r}(-i\, \ep_{\mu})^j\ =\ \left( \frac{1}{\widehat{M}_r}\right)^{(j)}(0),
$$
for every $j=0,1,...,r-1$, where $i$ denotes the imaginary unit. Then:
$$
\chi_r(x)\ :=\ \sum_{\mu=0}^{r-1}a_{\mu_r}\, M_r(t-\ep_{\mu}), \hskip1cm x \in \R,
$$
is a polynomial spline of order $r$, satisfying (\ref{null_moments}) and having support contained in $[\ep_0-r/2,\, \ep_{r-1}+r/2]$.
\end{theorem}
For instance, an example of kernel generated as in Theorem \ref{generation_kernels} which satisfy (\ref{null_moments}) with $r=2$ (see Fig. \ref{fig1}) is the following:
\be \label{chi2}
\chi_2(x)\ =\ 3\, M_2(x-2)\, -\, 2\, M_2(x-3), \hskip1cm x \in \R.
\ee
By procedures similar to that described by Theorem \ref{generation_kernels}, many other instances of kernels can be easily generated. For more details, and for other examples of kernels, see e.g., \cite{BUNE,BUST1,BAKAVI1,COSP3,BECOGA1}.
\begin{figure}
\centering
\includegraphics[scale=0.4]{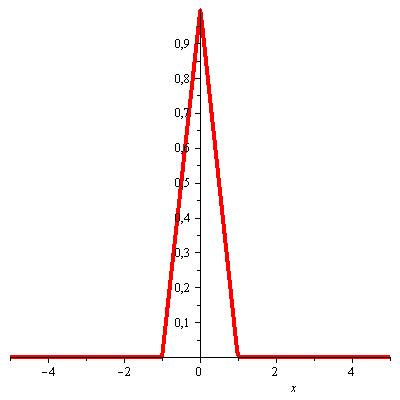}
\hskip0.4cm
\includegraphics[scale=0.4]{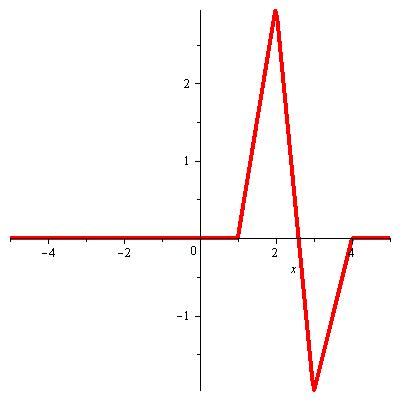}
\caption{{\small On the left: plot of the central B-spline of order 2 (roof-function). On the right: plot of the kernel $\chi_2$ defined in (\ref{chi2}).}} \label{fig1}
\end{figure}
Moreover, for the reconstruction of signals in terms of splines using finite number of samples from the past, see \cite{BUST1,COMIVI1}.
%

%%---------------------------------------------

\section{Conclusions}

By using the representation formula proved in Theorem \ref{th1} we are able to obtain an inverse result for the sampling Kantorovich operators. In particular, we show that the best order of approximation that can be achieved for the aliasing error $\| S_w f - f\|_{\infty}$ is ${\cal O}(w^{-1})$, as $w \to +\infty$, for $f \in C^{(2)}(\R)$ (Theorem \ref{thC2}). A similar result has been achieved on the space $C(\R)$, as in Theorem \ref{th2}.

Even if the above representation formula link the sampling Kantorovich operators with the generalized sampling ones of $f$ and its derivatives, the proof of the above inverse result cannot be directly reconnected to the corresponding one for the generalized operators. Indeed, for the operators $S_w f$ is not possible to establish an higher order of approximation theorem which revealed to be crucial for the proof of the aforementioned inverse result of \cite{BUST1} relative to $G_w$.

%%%%%%%%%%%%%%%%%%%%%%%%%%%

\vskip0.2cm

\section*{Acknowledgments}

The authors are members of the Gruppo  
Nazionale per l'Analisi Matematica, la Probabilit\'a e le loro  
Applicazioni (GNAMPA) of the Istituto Nazionale di Alta Matematica (INdAM). 

\noindent The authors are partially supported by the "Department of Mathematics and Computer Science" of the University of Perugia (Italy). Moreover, the first author of the paper has been partially supported within the 2017 GNAMPA-INdAM Project ``Approssimazione con operatori discreti e problemi di minimo per funzionali del calcolo delle variazioni con applicazioni all'imaging''.

\vskip0.1cm
%
%

%%%%%%
\end{document}